\documentclass[a4paper,12pt]{amsart}

\usepackage{amsmath}
\usepackage{amssymb}
\usepackage{amsfonts}
\usepackage[all]{xy}

\newtheorem{thm}{Theorem}[section]
\newtheorem{prop}[thm]{Proposition}
\newtheorem{lem}[thm]{Lemma}
\newtheorem{cor}[thm]{Corollary}
\newtheorem{rem}[thm]{Remark}
\newtheorem{rems}[thm]{Remarks}
\newtheorem{defi}[thm]{Definition}
\newtheorem{exo}{\bf\large Exercice}

\newcommand{\R}{\mathbb{R}}

\newcommand{\C}{\mathbb{C}}
\newcommand{\ds}{\displaystyle}

\newcommand{\beq}{\begin{eqnarray}}
\newcommand{\eeq}{\end{eqnarray}}
\newcommand{\bq}{\begin{equation}}
\newcommand{\eq}{\end{equation}}
\newcommand{\beqn}{\begin{eqnarray*}}
\newcommand{\eeqn}{\end{eqnarray*}}
\newcommand{\bex}{\begin{exo}}
\newcommand{\eex}{\end{exo}}
\newcommand{\ben}{\begin{enumerate}}
\newcommand{\een}{\end{enumerate}}

\newcommand{\Sum}{\displaystyle \sum}

\newcommand{\Int}{\displaystyle \int}
\newcommand{\Frac}{\displaystyle \frac}

\newcommand{\Sup}{\displaystyle \sup}
\newcommand{\Lim}{\displaystyle \lim}
\newcommand{\Liminf}{\displaystyle \liminf}
\newcommand{\Limsup}{\displaystyle \limsup}
\newcommand{\Max}{\displaystyle \max}

\def\virgp{\raise 2pt\hbox{,}}
\let\wt=\widetilde

\makeatletter

\@addtoreset{equation}{section}
\makeatother

\author{Mohamed Khalil Zghal}
\address{Universit\'e de Tunis El Manar, Facult\'e des Sciences de Tunis, Laboratoire \'Equations aux d\'eriv\'ees partielles (LR03ES04), 2092 Tunis, Tunisie}
\email{\sl zghal-khalil@hotmail.fr}

\title[Sharp Adams-type inequality invoking Hardy inequalities]{Sharp Adams-type inequality invoking Hardy inequalities}

\begin{document}
\begin{abstract}
We establish a sharp Adams-type inequality invoking a Hardy inequality for any even dimension. This leads to a non compact Sobolev embedding in some Orlicz space. We also give a description of the lack of compactness of this embedding in the spirit of \cite{Bahouri}.
\end{abstract}

\maketitle
\section{Introduction}
\subsection{Setting of the problem}
The Trudinger-Moser type inequalities have a long history beginning with the works of Pohozaev \cite{Poho} and Trudinger \cite{Trudinger}. Letting $\Omega\subset \R^n$ be a bounded domain with $n\geq 2$, the authors looked in these pioneering works for the maximal growth function $g:\R\rightarrow \R_+$ such that
$$ \sup_{u\in W_0^{1,n}(\Omega),\,\|\nabla u\|_{L^n}\leq 1}\Int_\Omega g(u)\,dx<+\infty\,,$$
 and they proved independently that the maximal growth is of exponential type. Thereafter, Moser improved these works by founding a sharp result known under the name of Trudinger-Moser inequality (see \cite{Moser}) and since that time, this subject has continued to interest researchers and Trudinger-Moser inequality has been extended in various directions (one can mention \cite{AT,Adams,Masmoudi,Ruf,Sani}) generating several applications. Among the results obtained concerning Trudinger-Moser type inequalities, we recall the so-called Adams' inequality in $\R^{2N}$.
 \begin{prop}\cite{Lam, Sani}\label{Sobolev}
 There exists a finite constant $\kappa>0$ such that
\begin{equation}\label{2}
\sup_{u\in H^N(\R^{2N}),\,\|u\|_{H^N(\R^{2N})}\leq
1}\;\;\int_{\R^{2N}}\,\left({\rm e}^{\beta_N
|u(x)|^2}-1\right)\;dx:=\kappa\,,
\end{equation}
where $\beta_N=N!\, \pi^N2^{2N}$, and for any $\beta>\beta_N$
 \begin{equation}\label{3}
\sup_{u\in H^N(\R^{2N}),\|u\|_{H^N(\R^{2N})}\leq
1}\;\;\int_{\R^{2N}}\,\left({\rm e}^{\beta
|u(x)|^2}-1\right)\;dx=+\infty\,.
\end{equation}
\end{prop}
\begin{rems}\quad\
\begin{itemize}
\item In the above proposition, the norm $\|\,.\,\|_{H^N}$ designates the following Sobolev norm
$$\|u\|_{H^N(\R^{2N})}^2:=\|u\|_{L^2(\R^{2N})}^2+\Sum_{j=1}^N\|\nabla^ju\|_{L^2(\R^{2N})}^2\,,$$
where $\nabla^ju$ denotes the j-th order gradient of $u$, namely
\begin{equation*}
 \nabla^ju=\left\{%
\begin{array}{ll}
\Delta^\frac{j}{2}u&\hbox{if $j$ is even},\\
\nabla\Delta^\frac{j-1}{2}u&\hbox{if $j$ is odd}.
\end{array}%
\right.
\end{equation*}
\medbreak
\item The proof of Proposition \ref{Sobolev}, treated firstly in the radial case and generalized then by symmetrization arguments, is based on the following Trudinger-Moser inequality in a bounded domain.
 \begin{prop}[\cite{Adams}, Theorem 1]\label{bounded d}
Let $\Omega$ be a bounded domain in $\R^{2N}$. There exists a positive constant $C_N$ such that
 \begin{equation*}
 \Sup_{u\in H_0^N(\Omega),\,\|\nabla^N u\|_{L^2}\leq 1}\Int_\Omega {\rm e}^{\beta_N|u(x)|^2}\;dx\leq C_N|\Omega|\,,
 \end{equation*}
where $|\Omega|$ denotes the Lebesgue measure of $\Omega$. Furthermore, this inequality is sharp.
 \end{prop}
 \medbreak
 \item As emphasized above, Proposition \ref{Sobolev} has been at the origin of numerous applications. Among others, one can mention the description of the lack of compactness of Sobolev embedding involving Orlicz spaces in \cite{ Bahouri,  Bahouri1', BP, IK, IK1}, the analysis of some elliptic and biharmonic equations in \cite{Sani1,Sani2,Sani3} and the study of global wellposedness and the asymptotic completeness for evolution equations with exponential nonlinearity in dimension two in \cite{Bahouri new1, Bahouri new2, BIP, Bahouri,Colliander,Ibrahim1,Ibrahim2}.
 \end{itemize}
\end{rems}
Sobolev embedding inferred by Proposition \ref{Sobolev} states as follows:
\begin{equation}\label{Orlicz}
H^N(\R^{2N})\hookrightarrow \mathcal{L}(\R^{2N})\,,
\end{equation}
where $\mathcal{L}$ is the so-called Orlicz space associated to the function $\phi(s):={\rm e}^{s^2}-1$ and defined as follows (for a complete presentation and more details, we refer the reader to \cite{Rao} and the references therein):
\begin{defi}
We say that a measurable function $u : \R^d\to\C$ belongs to
$\mathcal{L}(\R^d)$ if there exists $\lambda>0$ such that
$$
\Int_{\R^d}\left( {\rm e}^{\frac{|u(x)|^2}{\lambda^2}}-1\right)\;dx<\infty\,.
$$
We denote then
\begin{equation}\label{1}
\|u\|_{\mathcal{L}(\R^d)}=\inf\,\left\{\,\lambda>0,\Int_{\R^d}\left( {\rm e}^{\frac{|u(x)|^2}{\lambda^2}}-1\right)\;dx\leq
1\,\right\}\,.
\end{equation}
\end{defi}
\begin{rems}\quad\
\begin{itemize}
\item It is easy to check that $\|\,.\,\|_\mathcal{L}$ is a norm on the $\C$-vector space $\mathcal{L}$ which is invariant under translations and oscillations.
     \medbreak
     \item One can also verify that the number $1$ in \eqref{1} may be replaced by any positive constant. This changes the norm $\|\,.\,\|_{\mathcal{L}}$ to an equivalent one.
         \medbreak
 \item In the sequel, we shall endow the   space ${\mathcal{L}(\R^{2N})}$  with the
norm $\|\cdot\|_{\mathcal{L}(\R^{2N})}$ where the number $1$ is replaced by
the constant $\kappa$ involved in Identity~\eqref{2}. The Sobolev embedding \eqref{Orlicz} states then as follows:
\begin{equation}\label{inje}
    \|u\|_{\mathcal{L}(\R^{2N})}\leq\frac{1}{\sqrt{\beta_N}}\|u\|_{H^N(\R^{2N})}\,\virgp
\end{equation}
where the Sobolev constant $\ds \frac{1}{\sqrt{\beta_N}}$ is sharp.
\medbreak
\item
 Denoting by $L^{\phi_p}$ the Orlicz space associated to $\phi_p(s):={\rm{e}^{s^2}}-\Sum_{k=0}^{p-1} \frac{s^{2k}}{k!}\virgp$ with $p$ an integer larger than $1$, we deduce from Proposition \ref{Sobolev} the more general Sobolev imbedding
\begin{equation}\label{fi p}
H^N(\R^{2N})\hookrightarrow L^{\phi_p}(\R^{2N})\,.
\end{equation}
\medbreak
\item  Let us finally observe that $\mathcal{L} \hookrightarrow L^p$ for every $2 \leq p <\infty$\,.
    \end{itemize}
\end{rems}

In this article, our goal is twofold. Firstly  obtain an analogue of Proposition~\ref{Sobolev} in the radial framework of a  functional space $\mathcal{H}(\R^{2N})$ closely related to Hardy inequalities, which will easily lead to the following Sobolev imbedding \begin{equation}\label{embed}
  \mathcal{H}_{rad}(\R^{2N})\hookrightarrow\mathcal{L}(\R^{2N})\,.
  \end{equation}
  Secondly describe the lack of compactness of \eqref{embed}, which could be at the origin of several applications as it has been the case by previous characterizations of defect of compactness of various Sobolev embeddings.

\medbreak \noindent More precisely, for any integer $N\geq 2$, the space we will consider in this paper is defined
as follows:
\begin{equation}\label{def1}
\mathcal{H}(\R^{2N}):=\left\{u\in H^{1}(\R^{2N}); \frac{\nabla u}{|\,.\,|^{N-1}}\in L^2(\R^{2N})\right\}\,\cdot
\end{equation}
In view of the well-known Hardy inequalities (see for instance  \cite{BCG, BC,Hardy1, Hardy2}):
\begin{equation}\label{Hardy}
\left\|\frac{ u}{|\,.\,|^s}\right\|_{L^2(\R^d)}\leq C_{d,s}\| u\|_{\dot{H}^s(\R^d)}\,,\quad \forall\,s\in\Big[0,\frac{d}{2}\Big[\,\virgp
\end{equation}
the Sobolev space $H^N(\R^{2N})$ continuously embeds in the functional space $\mathcal{H}(\R^{2N})$ endowed with the norm
$$\|u\|_{\mathcal{H}(\R^{2N})}^2=\|u\|_{H^{1}(\R^{2N})}^2+\left\|\frac{\nabla u}{|\,.\,|^{N-1}}\right\|_{L^2(\R^{2N})}^2\,\cdot$$
 Actually, as shown by
the example of
function $$x\longmapsto \log(1-\log|x| )\, \mathbf{1}_{B_1(0)}(x)\,,$$
with $B_1(0)$  the unit ball of $\R^{2N}$, the embedding of $H^N(\R^{2N})$  into  $\mathcal{H}(\R^{2N})$ is strict for every $N\geq 2$.
\medbreak \noindent For the convenience of the reader, the following diagram recapitulates the different embeddings including the spaces involved in this work.
\begin{equation*}
\xymatrix{}
      \xymatrix{
  H^N(\R^{2N})  \ar@{^{(}->}[r] \ar@{^{(}->}[d]\ar@{^{(}->}[rd] & \mathcal{L}(\R^{2N}) \\
     H^1(\R^{2N})\ar@{<-^{)}}[r]\ar@{^{(}->}[ru]^*[@]{\hbox to 0pt{\hss\txt{N=1}\hss}} & \mathcal{H}(\R^{2N})\ar@{^{(}->} [u]_{\mbox{radial case}}
  }
  \end{equation*}

\medbreak \noindent The interest we take to the space $\mathcal{H}$ is motivated  by the importance of Hardy inequalities in Analysis (among others, we can mention blow-up methods or the study of pseudo-differential operators with singular coefficients).

\subsection{Main results}
 The result we obtained concerning the sharp Adams-type inequality in the framework of the space $\mathcal{H}(\R^{2N})$ takes the following form:
\begin{thm}\label{thm1}
For any integer $N$ greater than $2$, there exists a finite constant $\kappa '>0$ such that
\begin{equation}\label{inequality d1}
\Sup_{u\in\mathcal{H}_{rad}(\R^{2N}),\,\|u\|_{\mathcal{H}(\R^{2N})}\leq 1}\Int_{\R^{2N}}\Big({\rm e}^{\gamma_N|u(x)|^2}-1\Big)\; dx:=\kappa '\,,
\end{equation}
 where $\gamma_N:=\Frac{4\pi^NN}{(N-1)!}\,\virgp$ and for any $\gamma>\gamma_N$
 \begin{equation}\label{sharpness d1}
\Sup_{u\in\mathcal{H}_{rad}(\R^{2N}),\,\|u\|_{\mathcal{H}(\R^{2N})}\leq 1}\Int_{\R^{2N}}\Big({\rm e}^{\gamma|u(x)|^2}-1\Big)\; dx=+\infty\,.
\end{equation}
\end{thm}
\begin{rems}\quad\
\begin{itemize}
\item Note that the optimal constant involved  in Identity \eqref{inequality d1} is different from that appearing in Identity \eqref{2}.
 \medbreak
 \item Usually, the proofs of Trudinger-Moser inequalities reduce to the radial framework
under symmetrization arguments. In particular, in dimension two this question
is achieved by means of Schwarz symmetrization (see [1]). The
key point in that process is the preservation of Lebesgue norms and the minimization
of energy.\medbreak\noindent Unfortunately, the quantities $\left\|\frac{\nabla u}{|\,.\,|^{N-1}}\right\|_{L^2(\R^{2N})}$ cannot be minimized under Schwarz symmetrization as shown by the example $u_k(x):=\varphi(|x|+k)$, where $\varphi \neq 0$ is a smooth compactly supported function. The fact that $u_k^*=\varphi$ shows that the control of $\left\|\frac{\nabla {u_k^*}}{|\,.\,|^{N-1}}\right\|_{L^2(\R^{2N})}$ by $\left\|\frac{\nabla u_k}{|\,.\,|^{N-1}}\right\|_{L^2(\R^{2N})}$ fails.
\medbreak
\item It is clear that, when the constant $1$ in \eqref{1} is replaced by $\kappa '$, Theorem \ref{thm1} implies the following radial continuous embedding
$$
\|u\|_{\mathcal{L}(\R^{2N})} \leq \frac{1}{\sqrt{\gamma_N}}\|u\|_{\mathcal{H}_{rad}(\R^{2N})}\,\virgp
$$
where
 the Sobolev constant $\ds \frac{1}{\sqrt{\gamma_N}}$ is optimal.

 \medbreak
\item Observe that due to the continuous embedding
 $$H^N(\R^{2N})\hookrightarrow\mathcal{H}(\R^{2N})\,,$$
 Theorem \ref{thm1}   can be viewed as a generalization of Proposition \ref{Sobolev} in the radial framework.
\end{itemize}
\end{rems}

As mentioned above,  our second aim in this paper is to describe the lack of compactness of the Sobolev embedding   \eqref{embed}. Actually, this  embedding  is non compact at least for two reasons. The first reason is a lack of compactness
at infinity,  as shown by the  example  $u_k(x)=\varphi(x+x_k)$ where
$0\neq\varphi\in{\mathcal D}$ and $|x_k|\to\infty$, which converges weakly  to $0$ in $\mathcal{H}(\R^{2N})$ and satisfies $\|u_k\|_{\mathcal{L}(\R^{2N})}=\|\varphi\|_{\mathcal{L}(\R^{2N})}$. The second
reason is of concentration-type  as illustrated by the following example derived by P.-L. Lions \cite{Lions1,Lions2}:
    \begin{equation}\label{fk}
 f_{k}(x)=\left\{%
\begin{array}{ll}
0&\hbox{if\quad $|x|\geq 1$}\,,\\\\
-\sqrt{\frac{2N}{k  \, \gamma_N }}\log|x| &\hbox{if\quad ${\rm{e}}^{-k}\leq |x|< 1$}\,,\\\\
\sqrt{\frac{2Nk}{\gamma_N}} & \hbox{if \quad$|x|< {\rm{e}}^{-k}$}\,.
\end{array}%
\right.
\end{equation}
Indeed, we have the following proposition the proof of which is postponed to Section~\ref{Appendix} for the convenience of the reader.
\begin{prop}
\label{f_alpha}
The sequence  $(f_k)_{k \geq 0}$ defined above converges weakly  to $0$ in $\mathcal{H}(\R^{2N})$ and satisfies
$$ \|f_k\|_{\mathcal{L}(\R^{2N})} \stackrel{k\to\infty}\longrightarrow\Frac{1}{\sqrt{\gamma_N}}\,\cdot $$

\end{prop}

    \medbreak
\noindent It will be useful later on to emphasize  that  $f_k$ can be recast under the following form:
 \begin{equation}\label{decomp}
 f_k(x)=\sqrt{\frac{2Nk}{\gamma_N}}\,{\mathbf L}\Big(-\frac{\log|x|}{k}\Big)\,\virgp
 \end{equation}
where \begin{equation*}
 {\mathbf L}(t)=\left\{%
\begin{array}{ll}
1&\hbox{if\quad $t\geq 1$}\,,\\
t &\hbox{if\quad $0\leq t<1$}\,,\\
0& \hbox{if \quad $t<0$}\,,
\end{array}%
\right.
\end{equation*}
and that
\begin{equation}\label{formula}\|f_k\|_{H^1(\R^{2N})}\stackrel{k\rightarrow\infty}\longrightarrow 0\quad\mbox{and}\quad \Big\|\frac{\nabla f_k}{|\,.\,|^{N-1}}\Big\|_{L^2(\R^{2N})}= \big\|{\mathbf L}'\big\|_{L^2(\R)}= 1\,.\end{equation}
   In order to state our second result in a clear way,  let us introduce some objects as in \cite{Bahouri}.
    \begin{defi} \label{object}
    We shall designate by a scale any sequence $\underline{\alpha}:=(\alpha_n)_{n \geq 0}$
of positive real numbers going to infinity and by a profile any function $\psi$ belonging to the set
$$
\mathcal {P}:=\Big\{\;\psi\in L^2(\R,{\rm e}^{-2Ns}ds);\quad \psi'\in
L^2(\R),\;\psi_{|]-\infty,0]}=0\,\Big\}\,.
$$
Two scales $\underline{\alpha}$, $\underline{\beta}$ are said orthogonal if
    $$
   \Big|\log\left(\frac{\beta_n}{\alpha_n}\right)\Big|\stackrel{n\to\infty}\longrightarrow \infty\,.
    $$
    \end{defi}

 \begin{rem}  Recall  that each profile $\psi \in \mathcal {P}$ belongs to the H\"older space $C^{\frac 1 2} (\R)$, and satisfies
\begin{equation}\label{behavpsi}
\frac{\psi(s)}{\sqrt{s}} \to 0 \quad\mbox{as}\quad s \to 0 \,.
\end{equation}  Indeed taking advantage of the fact that  $\psi' \in L^2(\R)$, we get for any $s_2 >  s_1$ $$ \Big|\psi(s_2)- \psi(s_1)\Big| = \Big|\int_{s_1}^{s_2}\,\psi'(\tau)\,d\tau\Big| \leq \sqrt{s_2- s_1}\left(\int_{s_1}^{s_2}\,\,\psi'^2(\tau)\,d\tau\right)^{1/2}  \,,$$
which ensures that $\psi \in C^{\frac 1 2} (\R)$ and implies \eqref{behavpsi} by taking $s_1= 0$.
\end{rem}

The result we establish in this paper highlights the fact that the lack of compactness of the Sobolev embedding    \eqref{embed} can be described in terms of generalizations of the example by Moser \eqref{fk} as follows:
\begin{thm}\label{thm2}
Let $(u_n)_{n \geq 0}$ be a bounded sequence in $\mathcal{H}_{rad}(\R^{2N})$ such that
\begin{equation} \label{H1}
u_n\rightharpoonup 0\,,
\end{equation}
 \begin{equation} \label{H2}
\limsup_{n\to\infty}\|u_n\|_{\mathcal{L}(\R^{2N})}=A_0 >0\,,\quad\mbox{and}
\end{equation}
\begin{equation} \label{H3}
\Lim_{R\rightarrow\infty}\Limsup_{n\rightarrow\infty}\Int_{|x|\geq R}|u_n(x)|^2\,dx=0\,.
\end{equation}
 Then, there
exist a sequence of pairwise orthogonal scales $(\underline{\alpha}^{(j)})_{j\geq 1}$ and a sequence of profiles $(\psi^{(j)})_{j\geq 1}$ such that up to a subsequence extraction, we have for all
$\ell\geq 1$,\begin{equation}\label{decomposition}
u_n(x)=\Sum_{j=1}^{\ell}\,\sqrt{\frac{2N\alpha_n^{(j)}}{\gamma_N}}\;\psi^{(j)}\left(\frac{-\log|x|}{\alpha_n^{(j)}}\right)+{\rm
r}_n^{(\ell)}(x)\,,
\end{equation}
with $\Limsup_{n\to\infty}\;\|{\rm
r}_n^{(\ell)}\|_{\mathcal{L}(\R^{2N})}\stackrel{\ell\to\infty}\longrightarrow 0\,.$ Moreover, we have the following stability estimate
\begin{equation*}
\Big\|\frac{\nabla u_n}{|\,.\,|^{N-1}}\Big\|_{L^2(\R^{2N})}^2=\Sum_{j=1}^\ell\big\|\psi^{(j)'}\big\|_{L^2(\R)}^2+\Big\|\frac{\nabla {\rm r}_n^{(\ell)}}{|\,.\,|^{N-1}}\Big\|_{L^2(\R^{2N})}^2+\circ(1)\,,\quad n\rightarrow \infty\,.
\end{equation*}
\end{thm}

\begin{rems}\quad\
\begin{itemize}
\item The hypothesis of compactness at infinity \eqref{H3} is crucial: it allows to avoid the loss of Orlicz norm at infinity.
 \medbreak
 \item Note that the elementary concentrations
 \begin{eqnarray}\label{gnj}
g_n^{(j)}(x):=\sqrt{\frac{2N\alpha_n^{(j)}}{\gamma_N}}\;\psi^{(j)}\left(\frac{-\log|x|}{\alpha_n^{(j)}}\right)\,,\end{eqnarray} involved in Decomposition  \eqref{decomposition}  are in  $\mathcal{H}_{rad}(\R^{2N})$ whereas a priori, they do  not belong to $H^N(\R^{2N})$.
\medbreak
 \item Actually, the lack of compactness  of $H^N(\R^{2N})\hookrightarrow {\mathcal L}(\R^{2N})$ was characterized in \cite{BP} by means of the following type of elementary concentrations:
 \begin{equation}  \label{el} f_n(x):=  \frac {C_{N}} {  \sqrt{ \alpha_n}} \int_{|\xi| \geq {\rm
1}}\frac{ {\rm e}^{ i \, (x-x_n) \cdot \xi}}{|\xi|^{ 2 N }}  \; {\varphi} \Big(\frac { \log |\xi|} {\alpha_n}\Big) \, d \xi\,,\end{equation}
with $(\alpha_n)_{n \geq 0}$ a scale in the sense of Definition \ref{object}, $(x_n)_{n \geq 0}$  a sequence of points in $\R^{2N}$ and  $ {\varphi}$ a function in $L^2(\R_+)$.
Note that (see Proposition 1.7 in  \cite{BP})
$$ f_n(x)= \wt C_N \,\sqrt{\alpha_n}\;\psi\Big(\frac{-\log|x|}{\alpha_n}\Big)\,+ {\rm t}_n(x)\,,$$
  with $\ds \psi(y)= \int^y_0  {\varphi}(t)  \, dt$ and $\|
{\rm t}_n\|_{\mathcal L(\R^{2N})} \stackrel{n\to\infty}\longrightarrow0$.
\medbreak
\item Arguing as in \cite{Bahouri}, we have the following result:
\begin{prop}
Let us consider
\begin{equation*}
g_n(x) := \sqrt{\frac{2N\alpha_n}{\gamma_N}}\;\psi\left(\frac{-\log|x |}{\alpha_n}\right)\,\virgp
\end{equation*}
with $\psi$ a profile and $(\alpha_n)_{n \geq 0}$ a scale.
Then
\begin{equation}\label{gn} \|g_n\|_{\mathcal{L}(\R^{2N})}\stackrel{n\rightarrow\infty}\longrightarrow
 \frac{1}{\sqrt{\gamma_N}}\,\max_{s>0}\;\frac{|\psi(s)|}{\sqrt{s}}\,\cdot
\end{equation}
\end{prop}
\begin{proof}
Setting
$L=\Liminf_{n\rightarrow\infty}\|g_n\|_{\mathcal{L}(\R^{2N})}\,,$ we have for any fixed $\varepsilon>0$ and any $n$ sufficiently large (up to a subsequence extraction)
$$\Int_{\R^{2N}}\Big({\rm e}^{\big|\frac{g_n(x)}{L+\varepsilon}\big|^2}-1\Big)\,dx\leq \kappa'\,.$$
Therefore, there exists a positive constant $C$ such that
$$\alpha_n\Int_0^{+\infty}{\rm e}^{2N\alpha_ns\Big[\frac{1}{\gamma_N(L+\varepsilon)^2} \big|\frac{\psi(s)}{\sqrt{s}}\big|^2 -1\Big]}\,ds\leq C\,.$$
Using the fact that $\psi$ is a continuous function, we deduce that
$$L+\varepsilon\geq\Frac{1}{\sqrt{\gamma_N}}\Max_{s>0}\Frac{|\psi(s)|}{\sqrt{s}}\,\virgp$$
which ensures that
$$L\geq\frac{1}{\sqrt{\gamma_N}}\Max_{s>0}\Frac{|\psi(s)|}{\sqrt{s}}\,\cdot$$
\medbreak
\noindent To end the proof of \eqref{gn}, it suffices to show that for any positive real number $\delta$, the following estimate holds
$$\Int_{\R^{2N}}\Big({\rm e}^{\big|\frac{g_n(x)}{\lambda}\big|^2}-1\Big)\,dx\stackrel{n\rightarrow\infty}\longrightarrow 0\,,$$
where $\lambda:=\Frac{1+\delta}{\sqrt{\gamma_N}}\Max_{s>0}\frac{|\psi(s)|}{\sqrt{s}}\,\cdot$

\medbreak \noindent Performing the change of variable $r= {\rm e}^{-\alpha_ns}$, we easily get
\begin{eqnarray}\label{int}
\Int_{\R^{2N}}\Big({\rm e}^{\big|\frac{g_n(x)}{\lambda}\big|^2}-1\Big)\,dx&=&\Frac{2\pi^N\alpha_n}{(N-1)!}\Int_0^\infty {\rm e}^{-2N\alpha_ns\big(1-\frac{1}{\gamma_N  \lambda^2}{\big|\frac{\psi(s)}{\sqrt{s}}\big|}^2\big)}\,ds\\
&-&\Frac{2\pi^N\alpha_n}{(N-1)!}\Int_0^\infty {\rm e}^{-2N\alpha_ns}\,ds\,.\nonumber
\end{eqnarray}
Recalling that
$$\frac{\psi(s)}{\sqrt{s}}\rightarrow 0\quad\mbox{as}\quad s\rightarrow0\,, $$
we infer that for any $\varepsilon>0$, there exists $\eta>0$ such that
$$\frac{1}{\gamma_N \lambda^2}{\Big|\frac{\psi(s)}{\sqrt{s}}\Big|}^2<\varepsilon \quad\mbox{for any}\quad 0\leq s<\eta\,.$$
According to \eqref{int}, this gives rise to
\begin{eqnarray*}
\quad\Frac{2\pi^N\alpha_n}{(N-1)!}\Int_0^\eta {\rm e}^{-2N\alpha_ns\big(1-\frac{1}{\gamma_N\lambda^2}{\big|\frac{\psi(s)}{\sqrt{s}}\big|}^2\big)}\,ds&-&\Frac{2\pi^N\alpha_n}{(N-1)!}\Int_0^\eta {\rm e}^{-2N\alpha_ns}\,ds\\
&\leq& \frac{\pi^N\varepsilon}{N!(1-\varepsilon)}+\circ(1)\,,\quad n\rightarrow \infty\,,\end{eqnarray*}
which ensures the desired result.
\end{proof}
\medbreak
\item
Arguing as in Proposition 1.18 in \cite{Bahouri}, we get \begin{equation}\label{sup}
\Big\|\Sum_{j=1}^{\ell}\,g_n^{(j)}\Big\|_{\mathcal{L}(\R^{2N})}\stackrel{n\rightarrow\infty}\longrightarrow \sup_{1\leq
j\leq\ell}\,\left(\lim_{n\to\infty}\,\big\|g_n^{(j)}\big\|_{\mathcal{L}(\R^{2N})}\right)\,,
\end{equation}
where $g_n^{(j)}$ is defined by \eqref{gnj}.

\end{itemize}
\end{rems}
\subsection{Layout} The paper is organized as follows: Section 2 is devoted to the proof of the sharp Adams-type inequality in the framework of the space $\mathcal{H}_{rad}(\R^{2N})$, namely Theorem \ref{thm1}. In Section 3, we establish Theorem \ref{thm2} by describing the algorithm construction of the decomposition of a bounded sequence $(u_n)_{n \geq 0}$ in $\mathcal{H}_{rad}(\R^{2N})$, up a subsequence extraction, in terms of  asymptotically orthogonal
profiles in the spirit of the example by Moser. The last section is devoted to the proof of  Proposition~\ref{f_alpha}.

\medbreak
\noindent Finally, we mention that, $C$ will be used to denote a constant
which may vary from line to line. We also use $A\lesssim B$ to
denote an estimate of the form $A\leq C B$ for some absolute
constant $C$.
For simplicity, we shall also still denote by $(u_n)$ any
subsequence of $(u_n)$.

\section{Proof of the Theorem \ref{thm1}}
To establish Estimate \eqref{inequality d1}, we shall follow the 2D approach adopted in \cite{Ruf}  by setting for a fixed $r_0>0$ (to be chosen later on)
   \begin{equation*}
   I_1:=\Int_{B(r_0)}\Big({\rm e}^{\gamma_N|u(x)|^2}-1\Big)\; dx\,\quad \mbox{and}\quad
   I_2:=\Int_{\R^{2N}\backslash B(r_0)}\Big({\rm e}^{\gamma_N |u(x)|^2}-1\Big)\; dx\,,
   \end{equation*}
   where $B(r_0)$ denotes the ball centered at the origin and of radius $r_0$.\medbreak
\noindent The idea consists to show that it is possible to choose a suitable $r_0>0$ independently of $u$ such that $I_1$ and $I_2$ are bounded by a constant only depending on $r_0$ and $N$.

   \smallskip
 \noindent
   Let us start by studying the part $I_2$. Using the power series expansion of the exponential, we can write
   $$I_2=\Sum_{k=1}^\infty\frac{\gamma_N^k}{k!}\,I_{2,k}\,,\; \mbox{where}\quad  I_{2,k}:= \Int_{\R^{2N}\backslash B(r_0)}|u(x)|^{2k}\;dx\,.$$
   In order to estimate $I_{2,k}$, we take advantage  of the following radial estimate available for any function $u$ in $H^1_{rad}(\R^{2N})$ (for further details, see \cite{Sani}):
   \begin{equation}\label{radial estimate d}
   |u(x)|\leq \sqrt{\frac{(N-1)!}{\pi^N}}  \, \frac{ \|u\|_{H^1(\R^{2N})}}  {|x|^{N-\frac{1}{2}}}\quad \mbox{for a.e.} \quad x\in\R^{2N}\,,
   \end{equation}
 which  for  any integer $k\geq 2$, implies that
      \begin{eqnarray*}
   I_{2,k}&\leq& \left(\frac{(N-1)!}{\pi^N}\right)^{k}\|u\|_{H^1(\R^{2N})}^{2k}\frac{2\pi^N}{(N-1)!}\Int_{r_0}^\infty \frac{dr}{r^{(k-1)(2N-1)}}\\
   &\leq&\frac{2\pi^N}{(N-1)!} \left(\frac{(N-1)!}{\pi^N}\right)^{k}\|u\|_{H^1(\R^{2N})}^{2k}\frac{r_0^{k(1-2N)+2N}}{(2N-1)k-2N}\\
   &\leq&\frac{2\pi^N}{(N-1)!} \frac{r_0^{2N}}{2(N-1)} \left(\frac{(N-1)!}{\pi^N}\right)^{k}\|u\|_{H^1(\R^{2N})}^{2k}\frac{1}{r_0^{(2N-1)k}}\,\cdot
   \end{eqnarray*}
This gives rise to
\begin{eqnarray*}
I_2&\leq& \gamma_N\|u\|_{L^2(\R^{2N})}^2+\frac{2\pi^N}{(N-1)!} \frac{r_0^{2N}}{2(N-1)}\Sum_{k=2}^\infty \frac{1}{k!}\left(\frac{\gamma_N(N-1)!}{\pi^N}\frac{\|u\|_{H^1(\R^{2N})}^{2}}{r_0^{(2N-1)}}\right)^k\\
&\leq& \gamma_N+\frac{2\pi^N}{(N-1)!} \frac{r_0^{2N}}{2(N-1)}\Sum_{k=2}^\infty \frac{1}{k!}\left(\gamma_N\frac{(N-1)!}{\pi^N}\frac{1}{r_0^{(2N-1)}}\right)^k\,\virgp
\end{eqnarray*}
 under the fact that $\|u\|_{\mathcal{H}(\R^{2N})}\leq 1,$ which ensures that $I_2$ is bounded by a constant only dependent of  $r_0$ and $N$.
\medbreak\noindent
 In order to estimate $I_1$, we shall make use of the following Adams-type inequality, the proof of which is postponed at the end of the section.
 \begin{prop}\label{ball d}
There exists a constant $C_N>0$ such that for any positive real number $R$, we have
\begin{equation*}
\Sup_{u\in\big( \mathcal{H}_{rad}\cap H^1_0\big)(B(R)),\,\big\|\frac{\nabla u}{|\,.\,|^{N-1}}\big\|_{L^2}\leq 1}\Int_{B(R)} {\rm e}^{\gamma_N|u(x)|^2}\;dx\leq C_NR^{2N}\,,
\end{equation*}
 and this inequality is sharp.
\end{prop}
  Let us admit this proposition for the time being, and continue the proof of the theorem. The key point consists to associate to a function $u$ in $\mathcal{H}_{rad}(B(r_0))$ with $\|u\|_{\mathcal{H}(\R^{2N})}\leq 1$ an auxiliary function $w\in \big(\mathcal{H}_{rad}\cap H^1_0\big)(B(r_0))$ such that
$$\left\|\frac{\nabla w}{|\,.\,|^{N-1}}\right\|_{L^2(B(r_0))}\leq 1\quad\mbox{and}\quad u^2\leq w^2+d(r_0)\,,$$
where the function $d(r_0)>0$  depends only on $r_0$.
To this end, let us first emphasize that if $u$ belongs to $\mathcal{H}_{rad}(B(r_0))$ and satisfies $\|u\|_{\mathcal{H}(\R^{2N})}\leq 1\,,$ then $u$ is continuous far away from the origin. Indeed, for any real numbers $r_2>r_1>0\,,$ writing
$$u(r_2)-u(r_1)=\Int_{r_1}^{r_2}u'(s)\;ds\,,$$
we get by Cauchy-Schwarz inequality
\begin{eqnarray*}
|u(r_2)-u(r_1)|&\leq& \Big(\Int_{r_1}^{r_2}|u'(s)|^2s^{2N-1}\;ds\Big)^\frac{1}{2}\Big(\Int_{r_1}^{r_2}s^{- (2N-1)}\;ds\Big)^\frac{1}{2}
\\
&\leq& C\|\nabla u\|_{L^2(\R^{2N})} \Big(\Int_{r_1}^{r_2}s^{- (2N-1)}\;ds\Big)^\frac{1}{2}\,,
\end{eqnarray*}
which leads to the result. Thus, for any $0<r<r_0$, we can define the function $$v(r):=u(r)-u(r_0)\,,$$
which clearly  belongs to $\big(\mathcal{H}_{rad}\cap H^1_0\big)(B(r_0))\,.$
 In light of the radial estimate \eqref{radial estimate d}, this implies that
\begin{eqnarray*}
u^2(r)&\leq&v^2(r)+v^2(r)u^2(r_0)+1+u^2(r_0)\\
&\leq& v^2(r)+v^2(r) \frac{(N-1)!}{\pi^N} \frac{\|u\|_{H^1(\R^{2N})}^2}{r_0^{2N-1}}+1+ \frac{(N-1)!}{\pi^N} \frac{\|u\|_{H^1(\R^{2N})}^2}{r_0^{2N-1}}\\
&\leq& v^2(r)\left(1+  \frac{(N-1)!}{\pi^N} \frac{\|u\|_{H^1(\R^{2N})}^2}{r_0^{2N-1}}\right)+d(r_0)\,,
\end{eqnarray*}
where $d(r_0):=1+ \Frac{(N-1)!}{\pi^N} \Frac{\|u\|_{H^1(\R^{2N})}^2}{r_0^{2N-1}}\,\cdot$

\smallskip
\noindent Now by construction, the function
$$w(r):=v(r)\sqrt{1+  \frac{(N-1)!}{\pi^N} \frac{\|u\|_{H^1(\R^{2N})}^2}{r_0^{2N-1}}}\,\virgp$$
 belongs to $\big(\mathcal{H}_{rad}\cap H^1_0\big)(B(r_0))$, and easily satisfies
\begin{eqnarray*}
\Int_{B(r_0)}\frac{|\nabla w(x)|^2}{|x|^{2(N-1)}}\;dx&=& \left(1+  \frac{(N-1)!}{\pi^N} \frac{\|u\|_{H^1(\R^{2N})}^2}{r_0^{2N-1}}\right)\Int_{B(r_0)}\frac{|\nabla u(x)|^2}{|x|^{2(N-1)}}\;dx\\
&\leq&\left(1+\frac{(N-1)!}{\pi^N} \frac{\|u\|_{H^1(\R^{2N})}^2}{r_0^{2N-1}}\right) \left(1-\|u\|_{H^1(\R^{2N})}^2\right)\leq 1\,,
 \end{eqnarray*}
 provided that  $ \Frac{\pi^N} {(N-1)!}\,r_0^{2N-1}\geq 1$.

 \smallskip
\noindent  Applying Proposition \ref{ball d} with $r_0$ fixed so that $ \Frac{\pi^N} {(N-1)!}\,r_0^{2N-1}\geq 1$, we deduce that
$$ I_1\leq {\rm e}^{\gamma_N d(r_0)}\int_{B(r_0)}{\rm e}^{\gamma_N|w(x)|^2}\;dx\leq C_N\,{\rm e}^{\gamma_Nd(r_0)}\, r_0^{2N}\,,$$
which ensures the desired estimate, up to the proof of Proposition \ref{ball d}.

\medskip

To achieve the proof of  Identity \eqref{inequality d1},
let us then establish Proposition \ref{ball d}. To this end, let us for a function $u$ in $\big(\mathcal{H}_{rad}\cap H^1_0\big)(B(R))$ satisfying $\left\|\frac{\nabla u}{|\,.\,|^{N-1}}\right\|_{L^2(\R^{2N})}\leq 1$,  denote  by $$I(R):=\Int_{B(R)}{\rm e}^{\gamma_N |u(x)|^2}\;dx\,.$$
Our aim is  to show that $$I(R)\leq C_N R^{2N} \quad \mbox{whenever} \quad\Frac{2\pi^N}{(N-1)!}\Int_0^R|v'(r)|^2r\;dr\leq 1\,.$$ For that purpose, let us  perform  the change of variable $s=r^N$, and introduce the function $w(s)=\sqrt{\Frac{N\pi^{N-1}}{(N-1)!}}\,v\big(s^\frac{1}{N}\big)$. Recalling that $\gamma_N=\Frac{4\pi^N N}{(N-1)!}\, \virgp$ we infer that
$$I(R)=\frac{2\pi^N} {(N-1)!}\Int_0^{R} {\rm e}^{\gamma_N  |v(r)|^2} r^{2N-1}\;dr=\frac{2\pi^N} {N!}\Int_0^{R^N}{\rm e}^{4\pi |w(s)|^2} s\;ds\,\quad\mbox{and}$$
$$\frac{2\pi^N}{(N-1)!}\Int_0^R|v'(r)|^2r\;dr= 2\pi \Int_0^{R^N}|w'(s)|^2s\;ds\,.$$
The conclusion stems then from the 2D   radial framework of Proposition \ref{bounded d}.

\medskip Now in order to  prove the sharpness of the exponent $\gamma_N$, let us consider the sequence $(f_k)$ defined by \eqref{fk}. Since according to \eqref{formula}, we have
 $$\|f_k\|_{\mathcal{H}(\R^{2N})}=1+\circ(1)\,,\quad\mbox{as}\;k\rightarrow\infty\,,$$
 we get for any $\gamma > \gamma_N$
\begin{eqnarray*}
\Int_{\R^{2N}}\Big({\rm e}^{\gamma \big|\frac{f_k(x)}{\|f_k\|_{\mathcal{H}(\R^{2N})}}\big|^2}-1\Big)\;dx&\geq& \frac{2\pi^N}{(N-1)!}\Int_0^{{\rm e}^{-k}}\Big ({\rm e}^\frac{ 2N k \gamma}{\gamma_N(1+\circ(1))}-1\Big) r^{2N-1}\;dr\\
&\geq& \frac{\pi^N}{N!}\Big({\rm e}^{2Nk\frac{\gamma-\gamma_N(1+\circ(1))}{\gamma_N(1+\circ(1))}}-{\rm e}^{-2Nk}\Big)\stackrel{k\rightarrow \infty}\longrightarrow\infty\,,
\end{eqnarray*}
which ends the proof of the theorem.

\medskip

 \noindent

\section{Proof of Theorem \ref{thm2}}
\subsection{Scheme of the proof}
The proof of Theorem \ref{thm2}  relies on a diagonal subsequence extraction and uses in a
crucial way the radial setting and particularly the fact that we
deal with bounded functions far away from the origin. The heart of
the matter is reduced to the proof of the following lemma: \begin{lem}\label{heart}
Let $(u_n)_{n \geq 0}$ be a bounded sequence in $\mathcal{H}_{rad}(\R^{2N})$ satisfying Assumptions \eqref{H1}, \eqref{H2} and \eqref{H3}. Then there exist a scale $(\alpha_n)_{n \geq 0} $ and a profile $\psi$ in the sense of Definition \ref{object}, such that
\begin{equation}\label{prof}
\|\psi'\|_{L^2(\R)}\geq C_N A_0\,,
\end{equation}
where $C_N$ is a constant depending only on $N$.
 \end{lem} Inspired by the strategy developed in \cite{Bahouri}, the proof  is done in three steps. In the first
step, according to Lemma \ref{heart}, we extract the first scale
and the first profile satisfying Inequality \eqref{prof}. This
reduces the problem to the study of the remainder term. If the limit of its Orlicz norm is null we stop the process. If not, we prove
that this remainder term satisfies the same properties as the
sequence start which allows us to  extract a second  scale and a second
profile which verifies the above key property \eqref{prof}, by following  the lines of reasoning
 of the first step.  Thereafter, we establish the property of orthogonality
between the two first scales. Finally, we prove that this process converges.

\medbreak
\subsection{Extraction of the first scale and the first profile}
Let us consider a bounded sequence $(u_n)_{n \geq 0}$ in $\mathcal{H}_{rad}(\R^{2N})$ satisfying the assumptions of  Theorem \ref{thm2}, and let us set $v_n(s):= u_n({\rm e}^{-s})$. Then,  we have the following lemma:
\begin{lem}
Under the above assumptions, the sequence $(u_n)_{n \geq 0}$ converges strongly to $0$ in
$L^2(\R^{2N})$, and we have   for any real number $M$,
\begin{equation}\label{vn}
\Lim_{n\rightarrow\infty}\|v_n\|_{L^\infty(]-\infty,M])}=0\,.
\end{equation}
\end{lem}
\begin{proof}
 Let us first observe that for any positive real number $R$, we have
$$ \|u_n\|_{L^2(\R^{2N})}=\|u_n\|_{L^2(|x|\leq R)}+\|u_n\|_{L^2(|x|> R)}\,.
$$ Now, invoking  Rellich's theorem and the Sobolev embedding of $\mathcal{H}(\R^{2N})$ into $H^1(\R^{2N})$,  we infer that the space $\mathcal{H}(|x|<R)$ is compactly
embedded in $L^2(|x|<R)$. Therefore,
$$ \limsup_{n\to\infty}\, \|u_n\|_{L^2(|x|<R)}\stackrel{n\rightarrow\infty}\longrightarrow 0\,.$$
Taking advantage of the hypothesis of the compactness at infinity \eqref{H3}, we deduce the strong convergence of the sequence $(u_n)_{n \geq 0}$  to $0$ in $L^2(\R^{2N})$.

\medbreak\noindent
Finally, \eqref{vn} stems from  the strong convergence to zero of $(u_n)_{n \geq 0}$ in $L^2(\R^{2N})$ and the following well-known radial estimate available for any function $u$ in $H^1_{rad}(\R^{2N})$:
\begin{equation*}
|u(x)|\leq \sqrt{\Frac{(N-1)!}{\pi^N}}\,\Frac{\|u\|_{L^2(\R^{2N})}^\frac{1}{2}\|\nabla u\|_{L^2(\R^{2N})}^\frac{1}{2}}{|x|^{N-\frac{1}{2}}}\,\,,\quad \mbox{for a.e.} \quad x\in\R^{2N}\,.
\end{equation*}
\end{proof}
\medbreak

Now, arguing as in the proof of Proposition 2.3 in \cite{Bahouri}, we deduce  the following result:
\begin{prop}
\label{step1} For any $\delta>0$, we have
\begin{equation}
\label{depart} \sup_{s\geq
0}\left(\Big|\frac{v_n(s)}{A_0-\delta}\Big|^2-(2N-1)s\right)\to\infty\,,\quad
n\to\infty\,.
\end{equation}
\end{prop}

A byproduct of the previous proposition is the following corollary:
\begin{cor}\label{alpha}
 Under the above notations, there exists a sequence $(\alpha_n^{(1)})_{n \geq 0}$ in $\R_+$
tending to infinity such that
\begin{equation}\label{alphan1}
4\,\Big|\frac{v_n(\alpha_n^{(1)})}{A_0}\Big|^2-(2N-1)\,\alpha_n^{(1)}\stackrel{n\rightarrow\infty}\longrightarrow\infty\,,
\end{equation}
and for $n$ sufficiently large, there exists a positive constant $C$ such that
\bq \label{bo}
\frac{ A_0}{2}\sqrt{(2N-1)\alpha_n^{(1)}}\leq| v_n(\alpha_n^{(1)})|\leq C
\sqrt{\alpha_n^{(1)}}+\circ(1)\,,
\eq
where $C=\sqrt{\frac{(N-1)!}{2\pi^N}}\Limsup_{n\rightarrow\infty}\Big\|\frac{\nabla u_n}{|\,.\,|^{N-1}}\Big\|_{L^2(\R^{2N})}\,.$
\end{cor}
\begin{proof}
In order to establish \eqref{alphan1}, let us consider the sequences
$$W_n(s) := 4\,\Big|\frac{v_n(s)}{A_0}\Big|^2-(2N-1)s \quad \mbox{and} \quad a_n := \Sup_{s\geq 0} W_n(s)\,.$$
By definition, there exists a positive sequence $(\alpha_n^{(1)})_{n \geq 0}$ such that
$$W_n(\alpha_n^{(1)})\geq a_n-\Frac{1}{n}\,\cdot$$
Now, in view of  \eqref{depart}, $a_n\stackrel{n\rightarrow\infty}\longrightarrow\infty$ and then
$W_n(\alpha_n^{(1)})\stackrel{n\rightarrow\infty}\longrightarrow\infty$. It remains to prove that
$\alpha_n^{(1)}\stackrel{n\rightarrow\infty}\longrightarrow\infty$. If not, up to a subsequence  extraction,
the sequence $(\alpha_n^{(1)})_{n \geq 0}$ is bounded and so is
$(W_n(\alpha_n^{(1)}))_{n \geq 0}$ by \eqref{vn}, which  yields a contradiction.

\medbreak\noindent
Concerning Estimate \eqref{bo}, the left hand side  follows directly from \eqref{alphan1}. Besides, for any positive real number $s$, we have
$$|v_n(s)|\leq \Big|v_n(0)+\Int_0^sv_n'(\tau)\,d\tau\Big|\leq |v_n(0)|+s^\frac{1}{2}\|v_n'\|_{L^2(\R)}\,,$$
which according to \eqref{vn} which implies that $v_n(0)\stackrel{n\rightarrow\infty}\longrightarrow0$, and the following straightforward equality
$$\|v_n'\|_{L^2(\R)}=\sqrt{\frac{(N-1)!}{2\pi^N}}\,\Big\|\frac{\nabla u_n}{|\,.\,|^{N-1}}\Big\|_{L^2(\R^{2N})}\,\virgp$$ gives the right hand side of  Inequality \eqref{bo}, and thus ends the proof of the result.
\end{proof}

\medbreak
Corollary \ref{alpha} allows to extract the first scale, it remains to extract the first profile. To do so, let us set
$$\psi_n(y)=\sqrt{\frac{\gamma_N}{2N\alpha_n^{(1)}}} \,v_n(\alpha_n^{(1)} y)\,\cdot$$
It will be useful later on to point out that,  in view of Property \eqref{vn}, $\psi_n(0)\stackrel{n\rightarrow\infty}\longrightarrow 0$.

\medbreak\noindent The following result summarize the main properties of the sequence $(\psi_n)_{n \geq 0}$:
\begin{lem} Under notations of Corollary \ref{alpha}, there exists a profile $\psi^{(1)}\in \mathcal{P}$ such that, up to a subsequence extraction
$$\psi_n'\rightharpoonup {\psi^{(1)}}'\; in\; L^2(\R) \quad and \quad\big\|{\psi^{(1)}}'\big\|_{L^2}\geq\frac{A_0}{2}\sqrt{\frac{2N-1}{2N}\,\gamma_N}\,\cdot$$
\end{lem}
\begin{proof}
Noticing that $\|\psi_n'\|_{L^2(\R)}=\Big\|\frac{\nabla u_n}{|\,.\,|^{N-1}}\Big\|_{L^2(\R^{2N})}$\,, we infer that the sequence $(\psi_n')_{n \geq 0}$ is bounded in $L^2(\R)$. Thus, up to a subsequence extraction, $(\psi_n')_{n \geq 0}$ converges weakly in $L^2(\R)$ to some function $g$. Let us now introduce the function
$$\psi^{(1)}(s):=\Int_0^sg(\tau)\,d\tau\,.$$
Our aim is then  to prove
that $\psi^{(1)}$ is a profile and that  $\ds\big\|{\psi^{(1)}}'\big\|_{L^2}\geq\frac{A_0}{2}\sqrt{\frac{2N-1}{2N}\,\gamma_N}\,\cdot$

\medbreak\noindent On the one hand, applying Cauchy-Schwarz inequality, we get $$\big{|}\psi^{(1)}(s)\big{|}=\Big|\Int_0^sg(\tau)\; d\tau\Big|\leq \sqrt{s} \|g\|_{L^2(\R)}\,,$$ which  ensures that $\psi^{(1)}\in L^2(\R_+,{\rm e}^{-2Ns}ds)\,.$

\medbreak\noindent
On the other hand, we have $\psi^{(1)}(s)=0$ for all $s\leq 0$. Indeed, using the fact that $$\|u_n\|_{L^2(\R^{2N})}^2=\big(\alpha_n^{(1)}\big)^2\Int_\R |\psi_n(s)|^2 {\rm e}^{-2N\alpha_n^{(1)}s}\;ds\,,$$
we obtain that
$$\Int_{-\infty}^0|\psi_n(s)|^2\;ds\leq\Int_{-\infty}^0|\psi_n(s)|^2 {\rm e}^{-2N\alpha_n^{(1)}s}\;ds\leq \Frac{1}{\big(\alpha_n^{(1)}\big)^2}\|u_n\|_{L^2(\R^{2N})}^2\,\virgp$$
which implies  that $(\psi_n)_{n \geq 0}$ converges strongly to zero in $L^2(]-\infty,0[)$, and thus  for almost all $s\leq 0$ (still up to the  extraction of a  subsequence).

\medbreak\noindent But, we have \begin{equation*}
\psi_n(s)-\psi_n(0)=\Int_0^s \psi'_n(\tau)\;d\tau\underset{n\rightarrow \infty}\longrightarrow \Int_0^s g(\tau)\; d\tau=\psi^{(1)}(s)\,, \end{equation*}
which, according to the fact that $\psi_n(0)\stackrel{n\rightarrow\infty}\longrightarrow0$, implies that
\begin{equation}\label{a}
\psi_n(s)\stackrel{n\rightarrow \infty}\longrightarrow \psi^{(1)}(s)\,,\quad\forall\;s\in \R\,.
\end{equation}
We deduce that ${\psi^{(1)}}_{|]-\infty,0]}=0\,, $ which completes the proof of the fact that $\psi^{(1)}\in\mathcal {P}$.

\medbreak\noindent
Finally in light of \eqref{bo}, we have $$\big|\psi^{(1)}(1)\big|\geq\Frac{A_0}{2}\sqrt{\frac{2N-1}{2N}\,\gamma_N}\,\cdot$$
Since
$$\big\|{\psi^{(1)}}'\big\|_{L^2(\R)}\geq \Int_0^1\big|{\psi^{(1)}}'(\tau)\big|\;d\tau=\big|\psi^{(1)}(1)\big|\,,$$
this gives rise to
$$\big\|{\psi^{(1)}}'\big\|_{L^2}\geq\frac{A_0}{2}\sqrt{\frac{2N-1}{2N}\,\gamma_N}\,\virgp$$
which ends the proof of the key lemma \ref{heart}.
\end{proof}

\medbreak
\subsection{Study of the remainder term and iteration}
Our concern is to iterate the previous process and to prove that the algorithmic construction converges. For that purpose, let us first consider the remainder term
\begin{equation}\label{def}
{\rm r}_n^{(1)}(x)=u_n(x)-g_n^{(1)}(x)\,,
\end{equation}
where $$g_n^{(1)}(x)=\sqrt{\frac{2N\alpha_n^{(1)}}{\gamma_N}}\,\psi^{(1)}\left(\frac{-\log|x|}
{\alpha_n^{(1)}}\right)\,\cdot$$
It can be easily proved that  $({\rm r}^{(1)}_n)_{n \geq 0}$ is a bounded sequence in $\mathcal{H}_{rad}(\R^{2N})$ satisfying \eqref{H1}, \eqref{H3} and  the following property:
\begin{equation}\label{r1}
\Lim_{n\rightarrow \infty}\Big\|\frac{\nabla {\rm r}_n^{(1)}}{|\,.\,|^{N-1}}\Big\|_{L^2(\R^{2N})}^2=\Lim_{n\rightarrow\infty}\Big\|\frac{\nabla
u_n}{|\,.\,|^{N-1}}\Big\|_{L^2(\R^{2N})}^2-\big\|{\psi^{(1)}}'\big\|_{L^2(\R)}^2\,\cdot
\end{equation}

Let us now define  $A_1=\underset{n\rightarrow\infty}{\limsup}\|{\rm r}_n^{(1)}\|_{\mathcal{L}(\R^{2N})}\,.$ If $A_1=0\,,$ we stop the process. If not, arguing as above, we prove that there exists a constant $C$ such that
\begin{equation}\label{extr}
    \frac{ A_1}{2}\sqrt{(2N-1)\,\alpha_n^{(2)}}\leq \big|\tilde{{\rm r}}_n^{(1)}\big(\alpha_n^{(2)}\big)\big|\leq C \sqrt{\alpha_n^{(2)}}+o(1)\,,
\end{equation}
where $\tilde{{\rm r}}_n^{(1)}(s)={\rm r}_n^{(1)}({\rm e}^{-s})$ and that
 there exist a scale $(\alpha_n^{(2)})$ satisfying the statement of Corollary \ref{alpha} with $A_1$ instead of $A_0$ and a profile $\psi^{(2)}$ in $\mathcal{P}$ such that
 $${\rm r}_n^{(1)}(x)=\sqrt{\frac{2N\alpha_n^{(2)}}{\gamma_N}}\, \psi^{(2)}\left(\frac{-\log|x|}{\alpha_n^{(2)}}\right)+{\rm r}_n^{(2)}(x)\,,$$
 with $\big\|{\psi^{(2)}}'\big\|_{L^2}\geq\Frac{A_1}{2}\sqrt{\Frac{2N-1}{2N}\,\gamma_N}$ and
 \begin{equation*}
\Lim_{n\rightarrow \infty}\Big\|\frac{\nabla {\rm r}_n^{(2)}}{|\,.\,|^{N-1}}\Big\|_{L^2(\R^{2N})}^2=\underset{n\rightarrow \infty}{\lim}\Big\|\frac{\nabla
{\rm r}_n^{(1)}}{|\,.\,|^{N-1}}\Big\|_{L^2(\R^{2N})}^2-\big\|{\psi^{(2)}}'\big\|_{L^2(\R)}^2\,.
\end{equation*}
 Moreover, we claim that $(\alpha_n^{(1)})$ and $(\alpha_n^{(2)})$ are orthogonal in the sense of Definition~\ref{object}. Otherwise, there exists a constant $C$ such that
$$\frac{1}{C}\leq \left|\frac{\alpha_n^{(2)}}{\alpha_n^{(1)}}\right|\leq C\,.$$
 Making use of Equality \eqref{def}, we get
$$\tilde{{\rm r}}_n^{(1)}\big(\alpha_n^{(2)}\big)=\sqrt{\frac{2N\alpha_n^{(1)}}{\gamma_N}}\left(\psi_n\left(\frac{\alpha_n^{(2)}}{\alpha_n^{(1)}}\right)
-\psi^{(1)}\left(\frac{\alpha_n^{(2)}}{\alpha_n^{(1)}}\right)\right)\,\cdot$$
This implies that, up to a subsequence extraction,
$$\underset{n\rightarrow\infty}{\lim}\sqrt{\frac{\gamma_N}{2N\alpha_n^{(1)}}}~\tilde{{\rm r}}_n^{(1)}\big(\alpha_n^{(2)}\big)=\underset{n\rightarrow \infty}{\lim}\left(\psi_n\left(\frac{\alpha_n^{(2)}}{\alpha_n^{(1)}}\right)-\psi^{(1)}\left(\frac{\alpha_n^{(2)}}
{\alpha_n^{(1)}}\right)\right)=0\,,$$
which is in contradiction with the left hand side of Inequality \eqref{extr}.\medbreak\noindent
Finally, iterating the process, we get at step $\ell$
$$
u_n(x)=\Sum_{j=1}^{\ell}\,\sqrt{\frac{2N\alpha_n^{(j)}}{\gamma_N}}\;\psi^{(j)}\left(\frac{-\log|x|}{\alpha_n^{(j)}}\right)+{\rm r}_n^{(\ell)}(x)\,,
$$
with $$
\limsup_{n\to\infty}\,\|{\rm r}^{(\ell)}_n\|_{\mathcal{H}(\R^{2N})}^2\lesssim 1 -A_0^2-A_1^2-\cdots -A_{\ell-1}^2\,.
$$
This implies that $A_\ell\to 0$ as $\ell\to\infty$ and ends the proof of the theorem.
\section{Proof of Proposition \ref{f_alpha}
}\label{Appendix}
This section is devoted to the proof of Proposition \ref{f_alpha}. Actually, the fact that the sequence $(f_k)_{k \geq 0}$ converges weakly to $0$ in $ \mathcal{H}_{rad}(\R^{2N})$ stems from straightforward computations, and the heart of the matter consists to show that
\begin{equation} \label{lionmos}\|f_k\|_{\mathcal{L}(\R^{2N})} \stackrel{k\to\infty}\longrightarrow\Frac{1}{\sqrt{\gamma_N}}\,\cdot \end{equation}
 Firstly, let us  prove that $\Liminf_{k\rightarrow \infty}\|f_k\|_{\mathcal{L}(\R^{2N})}\geq \frac{1}{\sqrt{\gamma_N}}\,\cdot$ For that purpose, let us consider $\lambda>0$ such that
    $$\Int_{\R^{2N}}\Big({\rm e}^{\big|\frac{f_k(x)}{\lambda}\big|^2}-1\Big)\,dx\leq \kappa '\,.$$
    By definition, this gives rise to $$\Int_{|x|\leq {\rm e}^{-k}}\Big({\rm e}^{\big|\frac{f_k(x)}{\lambda}\big|^2}-1\Big)\,dx\leq \kappa '\,,$$
  and thus consequently
    $$\frac{\pi^N}{N!}\Big({\rm e}^\frac{2Nk}{\gamma_N\lambda^2}-1\Big){\rm e}^{-2Nk}\leq \kappa '\,.$$
    We deduce that
    $$\lambda^2\geq \Frac{2Nk}{\gamma_N\log(1+\frac{N!}{\pi^N}\,\kappa '\,{\rm e}^{2Nk})}\underset{k\rightarrow\infty}\longrightarrow \frac{1}{\gamma_N}\,\virgp$$
    which ensures that
    $$\Liminf_{k\rightarrow \infty}\|f_k\|_{\mathcal{L}(\R^{2N})}\geq \frac{1}{\sqrt{\gamma_N}}\,\cdot$$
   Now the fact  that $\Limsup_{k\rightarrow \infty}\|f_k\|_{\mathcal{L}(\R^{2N})}\leq \frac{1}{\sqrt{\gamma_N}}$ derives from the following proposition  the proof of which is postponed at the end of this section:  \begin{prop}\label{Adachi}
    Let $\gamma\in\,]0,\gamma_N[$. A positive constant $C_{\gamma,N}$ exists such that
    \begin{equation}\label{Adachi eq}
    \Int_{\R^{2N}}\Big({\rm e}^{\gamma|u(x)|^2}-1\Big)\; dx\leq C_{\gamma,N}\|u\|_{L^2(\R^{2N})}^2\,,
    \end{equation}
    for any non-negative function $u$ belonging to $\mathcal{H}_{rad}(\R^{2N})$, compactly supported and satisfying $u(|x|):[0,\infty[\rightarrow\R$ is decreasing and $\Big\|\Frac{\nabla u}{|\,.\,|^{N-1}}\Big\|_{L^2(\R^{2N})}\leq 1\,.$ Besides, Inequality \eqref{Adachi eq} is sharp.
    \end{prop}
    Assume indeed for the time being that  the above proposition is true. Then, for any  fixed $\varepsilon>0$, there exists $C_{\varepsilon}>0$ such that
\begin{eqnarray*}
  \int_{\mathbb{R}^{2N}} \left({\rm e}^{(\gamma_N-\varepsilon)|f_{k}(x)|^2}-1\right) \,dx
 \leq C_{\varepsilon,N}\|f_k\|_{L^2(\R^{2N})}^2\,,\end{eqnarray*}
which  leads to the desired result, by virtue of  the convergence of $(f_k)$ to zero in $L^2(\R^{2N})$.

  \medskip To end the proof of Proposition  \ref{f_alpha}, it remains to establish Proposition \ref{Adachi} the
proof of which is inspired from the one of Theorem 0.1 in \cite{AT}.
\begin{proof}
     Let $u$ satisfying the assumptions of Proposition \ref{Adachi}. Then there exists a function $v:\R_+\to \R_+$ such that
    $$u(x)=v(r)\,,\quad |x|=r\,,$$
    $$v'(r)\leq 0\,,\quad \forall\, r\geq 0\,,\; \mbox{and}$$
    $$\exists\,r_0>0\; \mbox{such that}\;v(r)=0\,\quad \forall \, r\geq r_0\,.$$
    \medbreak
\noindent    Setting $w(t)=\sqrt{\gamma_N}\,v\big({\rm e}^{-\frac{t}{2}}\big)\,,$ we can notice that $w$ satisfies the following properties:
    \begin{equation}\label{w1}
    w(t)\geq 0\,,\quad\forall\, t\in\R\,,
    \end{equation}
    \begin{equation}\label{w2}
    w'(t)\geq 0\,,\quad\forall\, t\in\R\,,\;\mbox{and}
    \end{equation}
    \begin{equation}\label{w3}
    \exists\, t_0\in \R\;\mbox{such that}\;w(t)=0\quad \forall\,t\leq t_0\,.
    \end{equation}
    \medbreak
\noindent    Besides, we obtain by straightforward computations that
    \begin{equation}\label{w4}
    \|w'\|_{L^2(\R)}=\sqrt{N}\,\Big\|\Frac{\nabla u}{|\,.\,|^{N-1}}\Big\|_{L^2(\R^{2N})}\leq \sqrt{N}\,\virgp
    \end{equation}
    \begin{equation}\label{w5}
    \Int_{\R}|w(t)|^2\,{\rm e}^{-Nt}\,dt=4N\|u\|_{L^2(\R^{2N})}^2\,,\;\mbox{and}
    \end{equation}
    \begin{equation}\label{w6}
    \Int_{\R}\Big({\rm e}^{\frac{\gamma}{\gamma_N}|w(t)|^2}-1\Big)\,{\rm e}^{-Nt}\,dt=\Frac{(N-1)!}{\pi^N}\Int_{\R^{2N}}\Big({\rm e}^{\gamma|u(x)|^2}-1\Big)\,dx\,.
    \end{equation}
    Thus to prove \eqref{Adachi eq}, it suffices to show that for any $\beta$ belonging to $]0,1[$, there exists a positive constant $C_{\beta}$ such that
    \begin{equation}\label{but}
    \Int_{\R}\Big({\rm e}^{\beta|w(t)|^2}-1\Big)\,{\rm e}^{-Nt}\,dt\leq C_{\beta}\Int_{\R}|w(t)|^2\,{\rm e}^{-Nt}\,dt\,,
    \end{equation}
    where $w$ satisfies \eqref{w1}, \eqref{w2}, \eqref{w3} and \eqref{w4}. For that purpose, let us set
    $$T_0:=\sup\{t\in\R;\;w(t)\leq 1\}\in\,]-\infty,+\infty]$$
    and write
     $$\Int_{\R}\Big({\rm e}^{\beta|w(t)|^2}-1\Big)\,{\rm e}^{-Nt}\,dt=I_1+I_2\,,$$
     where
      $$I_1:=\Int_{-\infty}^{T_0}\Big({\rm e}^{\beta|w(t)|^2}-1\Big)\,{\rm e}^{-Nt}\,dt\quad\mbox{ and}\quad I_2:= \Int_{T_0}^{+\infty}\Big({\rm e}^{\beta|w(t)|^2}-1\Big)\,{\rm e}^{-Nt}\,dt\,.$$
      In order to estimate $I_1$, let us notice that for any $t\leq T_0$, $w(t)$ belongs to $[0,1]$.
      Using the fact that there exists a positive constant $M$ such that
      $${\rm e}^x-1\leq Mx\,,\quad \forall \,x\in[0,1]\,,$$
      we deduce that
      $$I_1\leq M\,\beta\Int_{-\infty}^{T_0}|w(t)|^2{\rm e}^{-Nt}\,dt\,.$$
       Let us now estimate $I_2$. By virtue of Cauchy-Schwarz inequality, we get for any $t\geq T_0$
      \begin{eqnarray*}
      w(t)&=&w(T_0)+\Int_{T_0}^tw'(\tau)\,d\tau\\
      &\leq& 1+\sqrt{t-T_0}\,\|w'\|_{L^2(\R)}\,.
      \end{eqnarray*}
      This implies, in view of \eqref{w4}, that
      $$w(t)\leq 1+\sqrt{(t-T_0)N}\,.$$
      In addition, using the fact that for any $\varepsilon>0$ there exists $C_\varepsilon>0$ such that
      $$1+\sqrt{s}\leq \sqrt{(1+\varepsilon)s+C_\varepsilon}\,,$$
      we deduce that for any $t\geq T_0$
      $$w(t)^2\leq (1+\varepsilon )(t-T_0)N+C_\varepsilon\,.$$
      As $\beta\in]0,1[$, we can choose $\varepsilon$ such that $\beta (1+\varepsilon)-1<0\,.$
       Hence,
      \begin{eqnarray*}
      I_2&\leq& \Int_{T_0}^{+\infty} {\rm e}^{\beta (1+\varepsilon)(t-T_0)N+\beta C_\varepsilon-Nt}\,dt\\
      &\leq&{\rm e}^{\beta C_\varepsilon-NT_0}\Int_{T_0}^{+\infty} {\rm e}^{(t-T_0)N[\beta (1+\varepsilon)-1]}\,dt\\
      &\leq& \Frac{{\rm e}^{\beta C_\varepsilon-NT_0}}{N[1-\beta (1+\varepsilon)]}\,\cdot
      \end{eqnarray*}
      Since $\Int_{T_0}^{+\infty}|w(t)|^2\,{\rm e}^{-Nt}\,dt\geq \Int_{T_0}^{+\infty}{\rm e}^{-Nt}\,dt=\Frac{{\rm e}^{-NT_0}}{N}\,\virgp$ we infer that
      $$I_2\leq \Frac{{\rm e}^{\beta C_\varepsilon}}{1-\beta (1+\varepsilon)}\Int_{T_0}^{+\infty}|w(t)|^2\,{\rm e}^{-Nt}\,dt\,.$$
      Now, setting $C_\beta=\max\Big\{M\beta, \Frac{{\rm e}^{\beta C_\varepsilon}}{1-\beta (1+\varepsilon)}\Big\}$, we get \eqref{but}. This ends the proof of Inequality \eqref{Adachi eq}.

             \medbreak
       \noindent Finally, note that the example by Moser $f_k$ defined by \eqref{fk} illustrates the sharpness of Inequality \eqref{Adachi eq}, since $\|f_k\|_{L^2(\R^{2N})} \stackrel{k\to\infty}\longrightarrow  0$ and     $$ \Int_{\R^{2N}}\Big({\rm e}^{\gamma_N|f_k(x)|^2}-1\Big)\, dx\geq \Int_{|x|<{\rm e}^{-k}}\Big({\rm e}^{\gamma_N|f_k(x)|^2}-1\Big)\, dx=\Frac{\pi^N}{N!}\big(1-{\rm e}^{-2Nk}\big)\underset{k\rightarrow\infty}\longrightarrow\Frac{\pi^N}{N!}\,\cdot$$

      \end{proof}

\noindent{\bf Acknowledgments.} We are very grateful  to Professor Hajer Bahouri and Professor Mohamed Majdoub for  interesting discussions and careful reading of the manuscript. A part of this work was done at {\sf LAMA}-{\sf Universit\'e Paris-Est Cr\'eteil} whose hospitality and support we gratefully acknowledge.

\end{document}